\documentclass[letterpaper, 11 pt, conference]{ieeeconf}

\usepackage{graphics} 
\usepackage{epsfig} 
\usepackage{times} 
\usepackage{amsmath} 
\usepackage{amssymb}  
\usepackage{subfigure}
\usepackage{url}
\newtheorem{theorem}{Theorem}

\usepackage{cite}
\newtheorem{assumption}[theorem]{Assumption}
\usepackage{color}
\usepackage{algorithm}

\newtheorem{corollary}[theorem]{Corollary}

\newtheorem{definition}[theorem]{Definition}

{ 
\newtheorem{remark}[theorem]{Remark}

}

\newcommand{\bi}{\begin{itemize}}
\newcommand{\ei}{\end{itemize}}
\newcommand{\bd}{\begin{displaymath}}
\newcommand{\ed}{\end{displaymath}}
\newcommand{\be}{\begin{eqnarray*}}
\newcommand{\ee}{\end{eqnarray*}}

\newcommand{\bK}{{\bf K}}
\newcommand{\mE}{\mathbb{E}}
\newcommand{\mU}{\mathbb{U}}
\DeclareMathOperator{\Tr}{Tr}

\IEEEoverridecommandlockouts
\begin{document}

\title{Sample Complexity for Nonlinear Dynamics}

\author{Yongxin Chen and Umesh Vaidya
\thanks{Y.~Chen is with the
School of Aerospace Engineering, Georgia Institute of Technology, Atlanta, GA; {yongchen@gatech.edu}}
\thanks{U. Vaidya is with the Department of Electrical and Computer Engineering, Iowa State University, Ames, IA. ; {ugvaidya@iastate.edu}}
\thanks{Umesh Vaidya will like to acknowledge financial support from Department of Energy grant  {DE-OE0000876}}
}

\maketitle

\begin{abstract}
We consider the identification problems for nonlinear dynamical systems. An explicit sample complexity bound in terms of the number of data points required to recover the models accurately is derived. Our results extend recent sample complexity results for linear dynamics.  Our approach for obtaining sample complexity bounds for nonlinear dynamics relies on a linear, albeit infinite dimensional, representation of nonlinear dynamics provided by Koopman and Perron-Frobenius operator. We exploit the linear property of these operators to derive the sample complexity bounds. Such complexity bounds will play a significant role in data-driven learning and control of nonlinear dynamics. Several numerical examples are provided to highlight our theory.
\end{abstract}


\section{Introduction}

Nonlinear dynamical models possess the capacity to represent a variety of real-world systems and have been employed in different areas such as automatic control, robotics, autonomy and so on. A most common approach to obtaining a nonlinear model is via the first principle, which requires a good understanding of the underlying physics. In many cases, this requirement is however not realistic. Therefore, a data-driven approach using generated data samples to build a nonlinear model is becoming more and more critical. This is known as the system identification problem in control theory. 

Compared to that of linear systems, the system identification problems for nonlinear dynamics are considerably more difficult. There have been many works on this topic, and many algorithms have been proposed \cite{GirBai10,Nel13}. Most of these works focus on the asymptotical performance of the algorithms, which copes with the situation when the amount of data available goes to infinity.
A critical question pertains to the data efficiency hasn't been adequately addressed yet. How many data points do we need to recover a dynamical model to a certain precision? 

It turns out that this question falls into the scope of sample complexity theory, which is a key mathematical tool in theoretical machine learning. This tool is used to analyze the performance guarantee of machine learning models. Many fundamental results have been established along this line in supervised learning \cite{Vap99}. This attempt is not so successful in reinforcement learning, especially when the state space is continuous as in most control applications. Recently, as the first step in this direction, \cite{DeaManMatRec17,LokVufSheDmi18} some sample complexity results for data-driven linear quadratic regulator problems. 

The purpose of this work is to establish sample complexity results for nonlinear dynamics. To achieve this goal, we use a linear operator theoretic framework involving transfer Koopman and Perron-Frobenius (P-F) operators for linear representation and modeling of a nonlinear system. Linear operator theoretic framework has attracted lot of attention lately from the theoretical and applied dynamical system communities \cite{Dellnitz_Junge,Mezic2000,froyland_extracting,Junge_Osinga,Mezic_comparison,Dellnitztransport,mezic2005spectral,Mehta_comparsion_cdc,Vaidya_TAC,Vaidya_CLM_journal,susuki2011nonlinear,mezic_koopmanism,mezic_koopman_stability,surana_observer,huang2018data,huang2018_feedbackstabilization}. One of the features that makes this approach attractive is its ability to approximate complicated and complex nonlinear dynamical system from time-series data. The basic idea behind the linear operator framework is to lift the nonlinear finite dimensional evolution of a dynamical system in the state space to linear albeit infinite dimensional evolution of functions in the functional space.   Various algorithms are proposed for the finite dimensional approximation of these linear operators \cite{dellnitz2002set,Mezic2000,DMD_schmitt,rowley2009spectral,EDMD_williams,huang2018data}. However, to the best of authors knowledge, the problem of deriving sample complexity results for these operators has not been addressed yet. The linear nature of these operators allows us to carry out sample complexity analysis similar to the one developed for the case of a linear system but in  the lifted functional space \cite{DeaManMatRec17,LokVufSheDmi18}. We believe that sample complexity for a nonlinear system will play a fundamental role in our understanding of reinforcement learning algorithms, one of the fast-growing area of machine learning. 

The paper is organized as follows. In Section \ref{sec:pre} we introduce linear operator framework involving Koopman and P-F operators. The result on sample complexity is presented in Section \ref{sec:results}. We provide several examples in Section \ref{sec:eg} to illustrate our results. This follows by a short concluding remark in Section \ref{sec:conclusion}.

\section{Preliminaries}\label{sec:pre}
In this section, we provide brief overview of the theory behind linear operator involving P-F and Koopman operator. For more details please refer to \cite{Mezic2000,Mehta_comparsion_cdc,mezic_koopmanism}. 

Consider a discrete-time dynamical system 
\[x_{t+1}=T(x_t)\]
where $T:X\to X\subset \mathbb{R}^n$ with $X$ assumed to be compact. Associated with this dynamical system are two linear operators namely Koopman and Perron-Frobenius operator are are defined as follows. 
\begin{definition}[P-F operator] Let $L_2(X)$ be the space of square integrable functions. Under the assumption that the mapping $T$ is invertible, the P-F operator $\mathbb{P}_T: L_2(X)\to L_2(X)$ is defined as follows.
\begin{eqnarray}
[\mathbb{P}_T g](x)=g(T^{-1}(x))\left |\frac{\partial T^{-1}}{\partial x}\right|,
\end{eqnarray}
where $|\cdot|$ stands for the matrix determinant. 
\end{definition}
\begin{remark} The P-F operator can also be defined without the restrictive invertibility assumption on the mapping $T$ on the space of measures. For more details on this please refer to \cite{mezic2005spectral}. 
\end{remark}
\begin{definition}[Koopman Operator] The Koopman operator $\mathbb{U}_T: L_2(X)\to L_2(X)$ is defined as
\begin{eqnarray}
[\mathbb{U}_T h](x)=h(T(x)).
\end{eqnarray}
\end{definition}

The P-F and Koopman operators are dual to each other in the sense that
\begin{eqnarray}
\int_X [\mathbb{P}_T g](x)h(x)dx=\int_X g(x)[\mathbb{U}_T h](x)dx.\label{duality}
\end{eqnarray}
The duality can be expressed compactly as 
\[
\left<\mathbb{U}_T h,g\right>=\left<h,\mathbb{P}_T g\right>.
\]
These definitions extends to the setting of random dynamical systems. Consider the random dynamical system 
\begin{eqnarray}
x_{t+1}=F(x_t,\xi_t),\label{rds}
\end{eqnarray}
where $\xi_0,\xi_1,\ldots$ are assumed to independent identical distributed random vectors. One case of particular interest is
\begin{eqnarray}\label{eq:Txi}
x_{t+1}=T(x_t)+\xi_t,
\end{eqnarray}
which is a deterministic system perturbed by random noise $\xi$.

Next we provide definitions for the P-F and Koopman operators for the random dynamical system (\ref{rds}). We will use the same notation for the representing these operators for the deterministic and random dynamical systems. 
\begin{definition}[P-F operator] 
The P-F operator $\mathbb{P}_F :L_2(X)\to L_2(X)$ for the random dynamical system (\ref{eq:Txi}) is defined as
\begin{eqnarray}
[\mathbb{P}_F g](x)=\int_X g(y)\rho(x-T(y))dy,\label{PF_rds}
\end{eqnarray}
where $\rho(\cdot)$ is the probability density of $\xi$.
\end{definition}
\begin{definition}[Koopman Operator] 
The Koopman operator $\mathbb{U}_F: L_2(X)\to L_2(X)$ is defined as
\begin{eqnarray}
[\mathbb{U}_F h](x)= \mathbb{E}_\xi[h(F(x,\xi))],\label{K_rds}
\end{eqnarray}
where the expectation is taken with respect to $\xi$. 
\end{definition}

Again the duality between the P-F and Koopman operator follows in the random setting and equality (\ref{duality}) is true for P-F and Koopman operator as defined in Eqs. (\ref{PF_rds})-(\ref{K_rds}). This duality between the Koopman and P-F operator is exploited to propose finite dimensional approximation of the P-F operator using numerical algorithm developed for the approximation of Koopman operator \cite{huang2018data}.

\section{Sample Complexity of Koopman and Perron-Frobenius Operators}\label{sec:results}
Linear operator theoretic framework involving P-F and Koopman operator provides a powerful tool for the representation, analysis, and design of nonlinear dynamical systems. Our objective in this section is to derive sample complexity results for the finite dimensional approximation of these linear operators. Although several algorithms are proposed for the finite dimensional approximation of the Koopman operators from time series data, the fundamental principle behind these different algorithms remains the same. Hence, the sample complexity results that we derive, using extended dynamic mode decomposition (EDMD) algorithm \cite{DMD_schmitt,rowley2009spectral,EDMD_williams} and its modification for the approximation of P-F operator,   should apply to other algorithms as well. 

For the finite dimensional approximation, let 
\[X=[x_1,\ldots, x_T],\;\;\;Y=[y_1,\ldots,y_T]\]
be data points generated by random dynamical system \eqref{rds} through experiments or simulations. Note that here $y_k=F(x_k,\xi_k)$. These data samples could be from a single trajectory, in which case $y_k=x_{k+1}$, or different trajectories. 

To establish a finite dimensional approximation of a Koopman operator we first choose a set of finite many basis functions 
\begin{eqnarray}
\Psi(\cdot)=[\psi_1(\cdot),\ldots, \psi_N(\cdot)]^\top.\label{basis}
\end{eqnarray}
A corresponding approximation of a Koopman operator is nothing but its projection on this basis. More specifically, if
\[
	[\mathbb{U}_F\psi_k](\cdot)=\sum_{j=1}^N {\bf K}_{jk} \psi_j(\cdot)+r_k(\cdot),\;\;\;k=1,\ldots,N
\]
for some matrix $\bK$ with $r_k(\cdot)$ being almost perpendicular to the linear span of $\Psi$ for each $k$, then we say $K$ is the approximation of the Koopman operator $\mU_F$ on the basis $\Psi$. When the basis functions are properly chosen, the error functions $r_k$ are usually small. Consequently, the matrix $\bK$ is a relatively accurate representation of the Koopman operator and therefore the underlying nonlinear dynamics.

There are two sources of error in the approximation of the infinite dimensional linear operators. The first source of error is due to finite choice of the basis function used in the projection. Apart from the cardinality, choice of the basis function itself should to be rich enough to accurately capture the dynamics. In particular, the choice could be directed by the fact the unknown eigenfunctions of the operator lies in the span of the basis functions. The physics of the problem such as continuity property or the non-locality or locality of the phenomena to be captured can be used in determining the choice and number of basis function. The second source of error arise due to finite length of data used in the approximation of the operator. In this paper we are interested in characterizing the error due to the finite data length. Since our focus in the present paper is the estimation error induced by limited data points, we shall make the following assumption. 
\begin{assumption}\label{assumption_closed} 
The action of the Koopman operator on the basis functions, $\Psi$, is closed, i.e.,
\[
[\mathbb{U}_F\psi_k](x)=\sum_{j=1}^N {\bf K}_{jk} \psi_j(x),\;\;\;k=1,\ldots,N
\]
for some constant coefficients ${\bf K}_{jk}$. 
\end{assumption}

Let $\varphi$ be any function in the span of $\Psi$, namely, 
\[
	\varphi(x)=\Psi(x)^\top \alpha
\]
for some vector $\alpha\in \mathbb{R}^N$. By definition, the function $\varphi$ will evolve under the action of Koopman operator as 


\[
[\mathbb{U}_F \varphi](x)=\mE_\xi\{\varphi(y)\}=\mE_\xi\{\Psi(y)\}^\top\alpha,
\]
where $y=F(x,\xi)$.
It follows that  
\[
	\mE_\xi\{\Psi(y)\}^\top\alpha=[\mathbb{U}_F \varphi](x)=\Psi(x)^\top{\bf K}\alpha,
\]
which says that the coordinate of $\mathbb{U}_F \varphi$ in the space spanned by $\Psi$ is $\bK \alpha$. This implies that applying Koopman operator $\mU_F$ on a function $\varphi$ is nothing but multiplying its coordinate $\alpha$ by $\bK$ on the left. 

To estimate the approximation $\bK$, we multiply the equation by $\Psi(x)$
\begin{eqnarray}
\Psi(x)\mE_\xi\{\Psi(y)\}^\top =\Psi(x)\Psi(x)^\top {\bf K}\label{eq1}
\end{eqnarray}
Since $x$ is fixed, this is same as 
	\begin{equation}
		\mE_\xi\{\Psi(x)\Psi(y)^T\} = \Psi(x)\Psi(x)^T \bK.
	\end{equation}
    Now taking expectation with respect to the initial condition $x$ , we obtain
    \begin{equation}
		\mE_x\mE_\xi\{\Psi(x)\Psi(y)^T\} = \mE_x\{\Psi(x)\Psi(x)^T \}\bK.
	\end{equation}
Let 
	\[
		\Sigma_1=\mE_{x,\xi}\{\Psi(x)\Psi(y)^T\}
	\] 
and 	
	\[
		\Sigma_0= \mE_x\{\Psi(x)\Psi(x)^T\},
	\] 
then $\Sigma_1=\Sigma_0\bK$ and consequently $\bK=\Sigma_1\Sigma_0^{-1}$. 

When only generated data samples $X, Y$ are available, we have
	\begin{equation}\label{eq:Psievolove}
	\Psi(y_t) = {\bf K}^T \Psi(x_t) + \delta_t,
	\end{equation}
where $\delta_t:=\Psi(y_t) - \mE \{\Psi(F(x_t,\xi_t))\}$ satisfies
	\[
		\mE\{\delta_t\} = 0.
	\]
We assume that $\mE\{\delta_{t,j}^2\}\le \Delta$. This is clearly true when the dynamics is of the form \eqref{eq:Txi} with $\xi$ having bounded variance.

Multiplying (the transpose of) all terms of \eqref{eq:Psievolove} with $\Psi(x_t)$ on the left and sum up them over $t$ gives
	\[
		\hat\Sigma_1 = \hat\Sigma_0 {\bf K} + R
	\]
with 
	\[
		\hat\Sigma_1 = \frac{1}{T} \sum_{t=1}^T \Psi(x_t)\Psi(y_t)^T, \quad \hat\Sigma_0 = \frac{1}{T} \sum_{t=1}^T \Psi(x_t)\Psi(x_t)^T,
	\]
and 
	\[
		R=\frac{1}{T}\sum_{t=1}^T \Psi(x_t)\delta_t^T.
	\]
Clearly, $\hat\Sigma_0$ is an unbiased estimation of $\Sigma_0$ and $\hat\Sigma_0\rightarrow \Sigma_0$ when $T\rightarrow \infty$. The same argument holds for $\hat\Sigma_1, \Sigma_1$. The error term $R$ has zero expectation, i.e., $\mE\{R\}=0$.
Hence, a least square estimator of $\bK$ is given by
	\begin{equation}\label{eq:Khat}
		\hat{\bf K} := \hat\Sigma_0^{-1} \hat\Sigma_1. 
	\end{equation}
This estimator is widely used in the Koopman operator literatures \cite{EDMD_williams}. 

As $T\rightarrow \infty$, $\hat\Sigma_0\rightarrow \Sigma_0, \hat\Sigma_1\rightarrow \Sigma_1$, and therefore $\hat\bK\rightarrow \bK$. 
In addition, there estimation error 
	\[
		\hat{\bf K} - {\bf K}= \hat\Sigma_0^{-1} R
	\]
can be analyzed as follows. 
We first invoke Cauchy-Schwarz inequality, which gives
	\begin{eqnarray*}
		\mE \{\|\hat{\bf K} - {\bf K}\|_F\} &=& \mE\{\|\hat\Sigma_0^{-1} R\|_F\}
		\\
		&\le& \mE\{\|\hat\Sigma_0^{-1}\|_F\|R\|_F\}
		\\
		&\le& \sqrt{\mE\{\|R\|_F^2\}\mE\{\|\hat\Sigma_0^{-1}\|_F^2\}}.
	\end{eqnarray*}
In the above, $\|\cdot\|_F$ denotes Frobenius norm. 
To attain an upper bound on $\mE\{\|R\|_F^2\}$, we observe that each element $R_{kj}$ of $R$ satisfies
	\begin{eqnarray*}
		\mE\{R_{kj}^2\} &=& \frac{1}{T^2} \sum_{t=1}^T \sum_{s=1}^T \mE\{\psi_k(x_t)\delta_{t,j} \psi_k(x_s)\delta_{s,j}\}
		\\
		&=&\frac{1}{T^2} \sum_{t=1}^T \mE\{\psi_k(x_t)^2\delta_{t,j}^2\}
		\\
		&\le&  \frac{\Delta}{T^2} \sum_{t=1}^T \mE\{\psi_k(x_t)^2\}.
	\end{eqnarray*}
The second equality follows from the fact that $\delta_t$ is conditionally independent of $x_s$ for all $s\le t$. The last inequality follows from the boundedness assumption $\mE\{\delta_{t,j}^2\}\le \Delta$.
Summing up the above over all $k, j$ we obtain
	\[
		\mE\{\|R\|_F^2\} \le \frac{\Delta}{T}\mE\{\Tr(\hat\Sigma_0)\}.
	\]
Therefore,
	\begin{equation}\label{eq:expectation}
		\mE \{\|\hat{\bf K} - {\bf K}\|_F\} \le \frac{\sqrt{\Delta}}{\sqrt{T}} \sqrt{\mE\{\Tr(\hat\Sigma_0)\}\mE\{\|\hat\Sigma_0^{-1}\|_F^2\}}.
	\end{equation}
\begin{theorem}\label{thm:main}
Let $\epsilon>0$ and $T> 2N+2$, then with probability at least $1-\epsilon$, the least square estimator $\hat{\bf K}$ in \eqref{eq:Khat} reconstructs ${\bf K}$ within a Frobenius norm error bounded by
	\begin{equation}
	\|\hat{\bf K} - {\bf K}\|_F \le \frac{\sqrt{\Delta}}{\epsilon \sqrt{T}} \sqrt{\mE\{\Tr(\hat\Sigma_0)\}\mE\{\|\hat\Sigma_0^{-1}\|_F^2\}}.
	\end{equation}
\end{theorem}
\begin{proof}
The assumption $T> 2N+2$ guarantees that the least square estimator $\hat{\bf K}$ in \eqref{eq:Khat} is well defined. The rest follows by applying Markov's inequality to the nonnegative random variable $\|\hat{\bf K} - {\bf K}\|_F$. 
\end{proof}

In \cite{Vaidya_CLM_journal}, the duality between the P-F and Koopman operator is exploited to provide algorithm for the finite dimensional approximation of the P-F operator. Following (\ref{duality}) and under the assumption that $g$ and $h$ lie in the span of $\Psi$ i.e., $g=\Psi^\top a$ and $h=\Psi^\top b$ for some constant vectors $a$ and $b$, we can write 
\[
\left<h,g\right>=a^\top\Lambda b,
\]
where $[\Lambda]_{ij}=\left<\psi_i,\psi_j\right>$ for $i=2,\ldots,N$ is a symmetric matrix. Using Assumption \ref{assumption_closed}, we have
\begin{eqnarray}
\left<\mathbb{U}_F h,g\right>=({\bf K}a)^\top \Lambda b=a^\top {\bf K}^\top \Lambda b=\left<h,\mathbb{P}_F g\right>.\label{eq}
\end{eqnarray}
Let ${\bf P}$ be the finite dimensional approximation of the P-F operator on the basis function, $\Psi$. Then using (\ref{eq}), we obtain
\[a^\top {\bf K}^\top \Lambda b=\left<h,\mathbb{P}_F g\right>=a^\top \Lambda {\bf P}b.\]
Since the above is true for all $g$ and $h$ in the span of $\Psi$, we obtain following finite dimensional approximation of the P-F operator in terms of the Koopman operator
\[
{\bf P}=\Lambda^{-1}{\bf K}^\top \Lambda.
\]
\begin{corollary}
Let $\epsilon>0$ and $T> 2N+2$, then with probability at least $1-\epsilon$, the least square estimator $\hat{\bf P}$ in \eqref{eq:Khat} reconstructs ${\bf P}$ within a Frobenius norm error bounded by
	\begin{equation}
	\|\hat{\bf P} - {\bf P}\|_F\! \le\! \frac{\sqrt{\Delta}\|\Lambda\|_2\|\Lambda^{-1}\|_2}{\epsilon \sqrt{T}} \sqrt{\mE\{\Tr(\hat\Sigma_0)\}\mE\{\|\hat\Sigma_0^{-1}\|_F^2\}}.
	\end{equation}
\end{corollary}
\begin{proof}
It follows directly from Theorem \ref{thm:main} and the definition of induced 2-norm $\|\cdot\|_2$. 
\end{proof}

\section{Numerical Examples}\label{sec:eg}
We provide two examples to illustrate our results. In the first one, the Assumption \ref{assumption_closed}  is valid. The second one is a standard Van der Pol  oscillator which doesn't satisfy this assumption.

\noindent{\bf Example 1}: Consider the following discrete time dynamical system
\begin{eqnarray}
x_1^{t+1}&=&\rho x_1^t+\xi_1^t,\nonumber\\
x_2^{t+1}&=&\mu x^t_2+(\rho^2-\mu)c (x^t_1)^2+\xi_2^t,
\end{eqnarray}
where $\xi_1, \xi_2$ are standard unit variance Gaussian noise, and $\rho<1, \mu<1, c>0$ are parameters.


It is easy to see that the action of the Koopman operator is closed for the basis functions,
\[
	\Psi(x)=[1, x_1,x_2, x_1^2]
\]
with
\[
	\bK=
	\left[\begin{matrix}
	1 & 0 & 0 & 1\\
	0 & \rho & 0 & 0\\
	0 & 0 & \mu & 0\\
	0 & 0 & (\rho^2-\mu)c & \rho^2
	\end{matrix}\right].
\]

Figure \ref{fig:eg1a} and Fig. \ref{fig:eg1b} showcase the estimation errors $err = \frac{\|\hat\bK-\bK\|_F}{\|\bK\|_F}$ as a function of time step $T$, which match with our results pretty well. Note that we used a single trajectory to estimate $\bK$ but the estimation errors are averaged over $50$ realizations. 
\begin{figure}[h]
\centering
\includegraphics[width=0.35\textwidth]{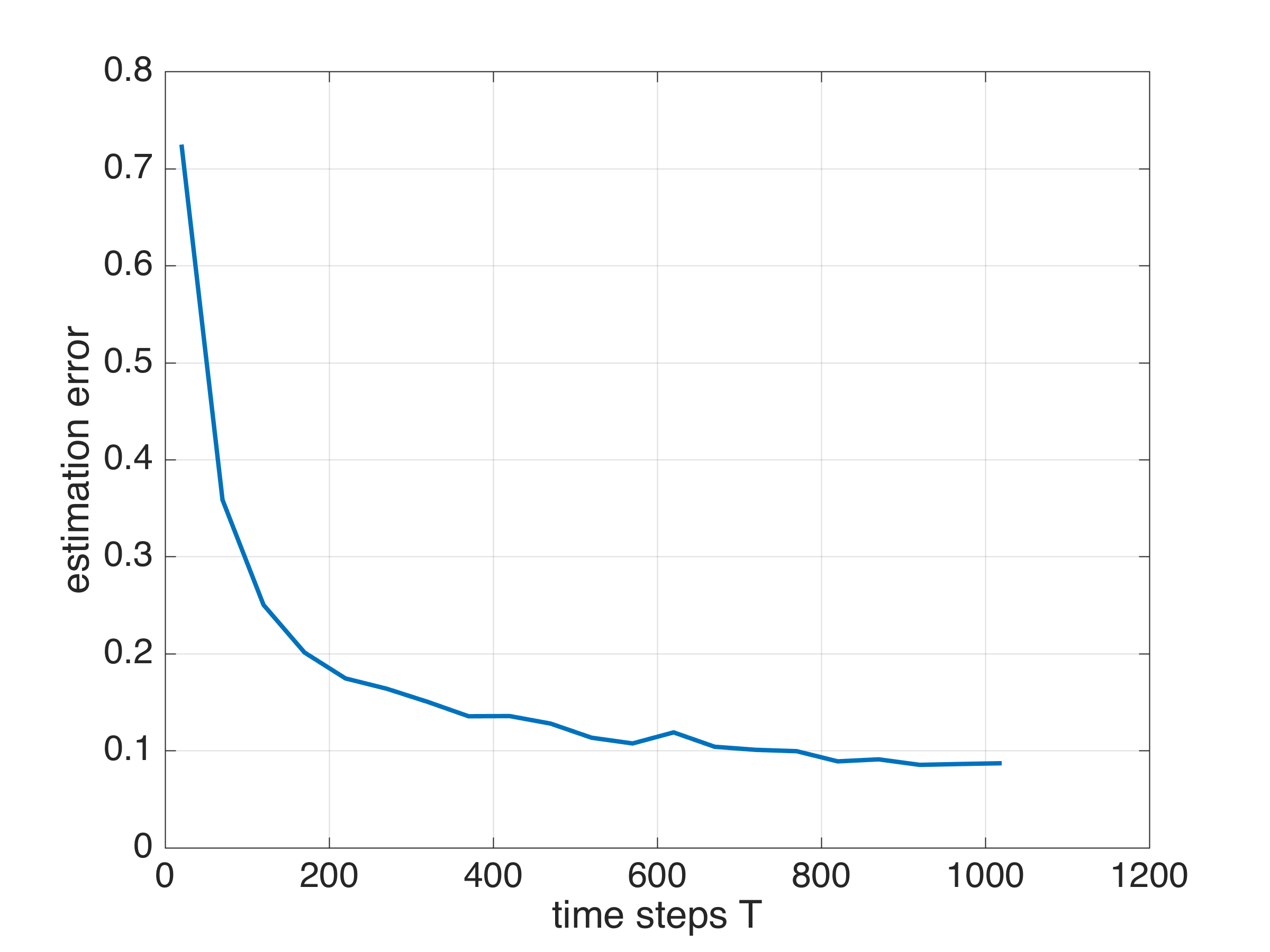}
\caption{Example 1: $\rho=0.2,\,\mu=0.3,\,c=1$}
\label{fig:eg1a}
\end{figure}
\begin{figure}[h]
\centering
\includegraphics[width=0.35\textwidth]{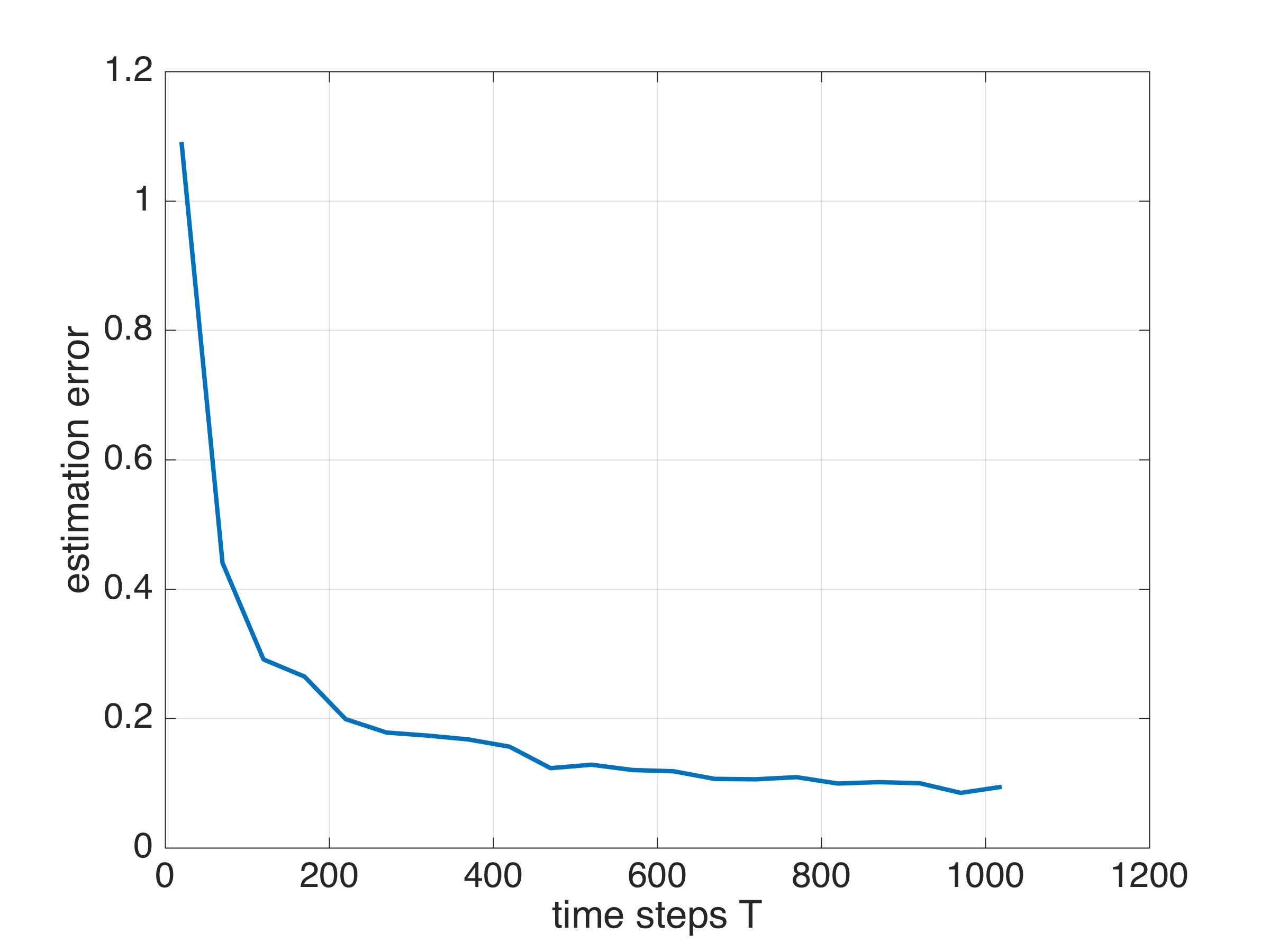}
\caption{Example 1: $\rho=0.8,\,\mu=0.8,\,c=0.9$}
\label{fig:eg1b}
\end{figure}

\noindent {\bf Example 2}: The second example that we consider is the discretized version of Van der Pol oscillator. The discretized equation for the Van der Pol  oscillator is given by
\begin{eqnarray}
x_1^{t+1}&=& x_1^t+\Delta (x_2^t)+\xi_2^t\nonumber\\
x_2^{t+1}&=&x_2^t+\Delta\left( (1-(x_1^t)^2)x_1^t-x_1^t\right)+\xi_1^t
\end{eqnarray}
where $\Delta$ is the time step of discretization and is chosen to be equal to $\Delta=0.0001$. Monomial with largest degree two is used as the choice of basis functions. Hence there are total of six functions in the basis. As can be seen from Fig. \ref{fig:eg2a}, even though the system is not closed with respect to Koopman operator, the convergent result matches the theory pretty well.
\begin{figure}[h]
\centering
\includegraphics[width=0.35\textwidth]{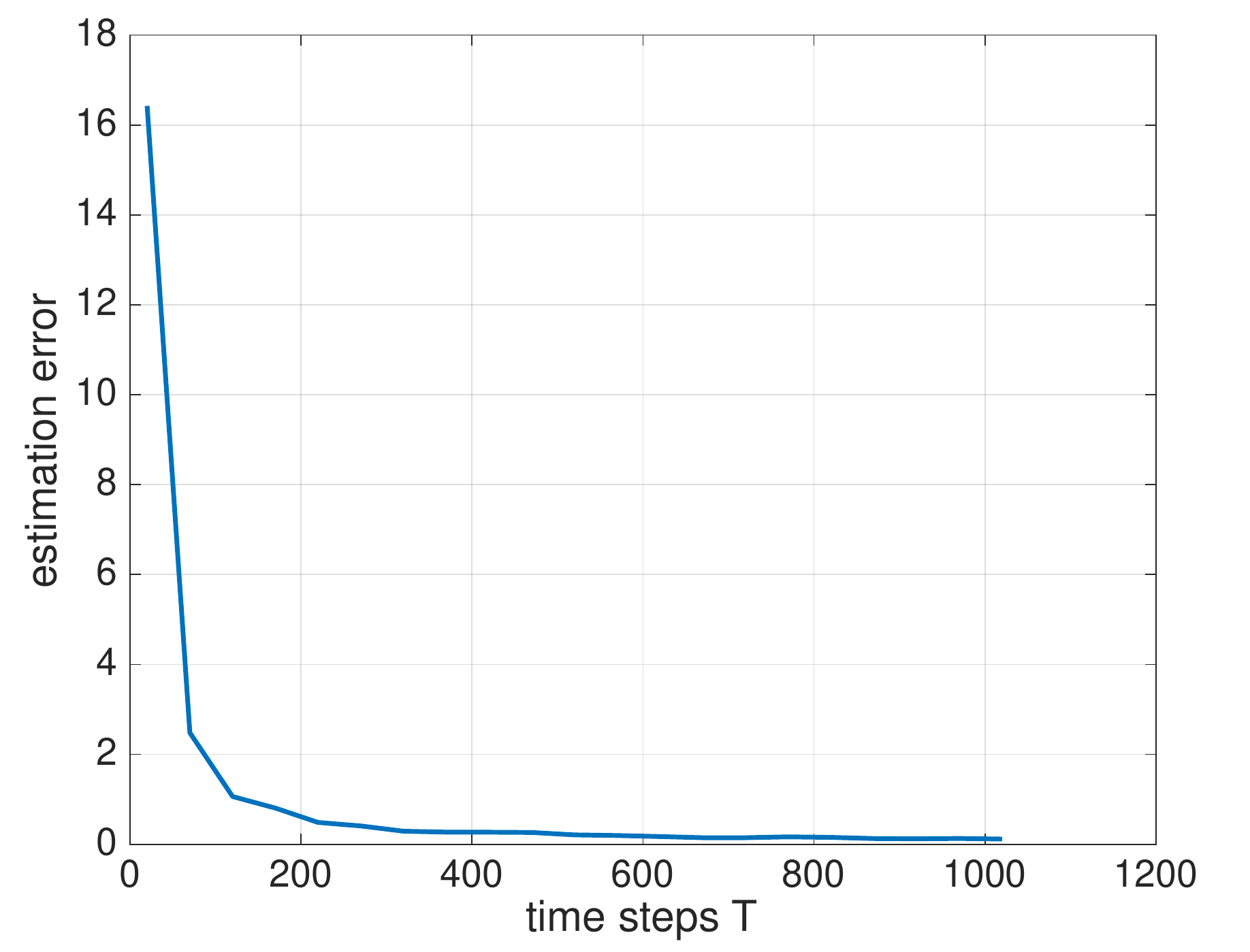}
\caption{Example 2: Van der Pol Oscillator}
\label{fig:eg2a}
\end{figure}

\section{Conclusion}\label{sec:conclusion}
We derived sample complexity results for the identification of nonlinear dynamical systems. The results make use of linear operator theoretic framework involving Koopman operator which lifts nonlinear systems to infinite dimensional linear systems. The results are derived for discrete-time dynamical systems but can be extended to continuous-time setting.


\bibliographystyle{IEEEtran}
\bibliography{refs,ref,ref1,reference}

\end{document}